\documentclass[11pt]{article}

\usepackage{amsfonts}
\usepackage{float, tikz, subfigure, array, caption}

\usepackage{mathrsfs}
\usepackage{graphics}

\usepackage{amssymb,amsthm,amsfonts}
 \usepackage{amsmath}
\def\dfrac{\displaystyle\frac}

\numberwithin{equation}{section}

\newtheorem{theorem}{Theorem}[section]

\newtheorem{lemma}{Lemma}[section]

\newtheorem{remark}{Remark}[section]

\def\bbr{{\mathbb R}}
\def\bbn{{\mathbb{N}}}

\newcommand{\coma}{\; , \;}
\newcommand{\dx}[1]{\, \mathrm{d} #1}

\newcommand{\mb}[1]{\mathbf{#1}}
\begin{document}

\title{Optimal estimate of field concentration between multiscale nearly-touching inclusions for 3-D Helmholtz system}

\author{ Youjun Deng\footnote{School of Mathematics and Statistics, HNP-LAMA, Central South University, Changsha, Hunan, China. Email:youjundeng@csu.edu.cn}, Yueguang Hu\footnote{Department of Mathematics, City University of Hong Kong, Hong Kong, China. Email: yueghu2-c@my.cityu.edu.hk}, Hongyu Liu\footnote{Department of Mathematics, City University of Hong Kong, Hong Kong, China. Email: hongyliu@cityu.edu.hk} and Wanjing Tang\footnote{School of Mathematics and Statistics, Central South University, Changsha, Hunan, China. Email: wanjingtang@csu.edu.cn}}

\maketitle

\begin{abstract}
We are concerned with the field concentration between two nearly-touching inclusions with high-contrast material parameters, which is a central topic in the theory of composite materials. The degree of concentration is characterised by the blowup rate of the gradient of the underlying field. In this paper, we derive optimal gradient estimates for the wave filed of the 3-D Helmholtz system in the quasi-static regime. There are two salient features of our results that are new to the literature. First, we cover all the possible scenarios that the size of the inclusions are in different scales in terms of the asymptotic distance parameter $\epsilon$, which can be used to characterise the curvature effects of the shape of the inclusions on the field concentration. Second, our estimates can not only recover the known results in the literature for the static case, but can also reveal the interesting frequency effect on the field concentration. In fact, a novel phenomena is shown that even if the static part vanishes, field blowup can still occur due to the (low) frequency effect.

\vspace{0.3cm}
\end{abstract}

\vspace{0.3cm}


\section{INTRODUCTION}

We are concerned with the field concentration between two nearly-touching inclusions with high-contrast material parameters, which is a central topic in the theory of composite materials. In the physical setup, the high-contrast inclusions signify the building blocks of the composite material, and the significant field concentration may cause failure of the composite structure or other practically important consequences. The degree of concentration is characterised by the blowup rate of the gradient of the underlying field. Many sharp gradient estimates have been established in the literature, mainly for the static case in different physical setups. Here, we mention a few results for the Laplace equation in different setups arising in understanding composite materials in conductivity and anti-plane elasticity, which are related to our study in the current article. It was proved that the gradient field is uniformly bounded  if the material parameter is bounded both below and above \cite{bonnetier2000elliptic,li2000gradient}.
When the material parameter takes extreme values,  the gradient field generically becomes unbounded as the asymptotic parameter $\epsilon$ approaches $0^+$. Here, $\epsilon\in\mathbb{R}_+$ signifies the distance between the underlying two close-to-touch high-constrast inclusions.
For perfect insulators, the optimal blowup rate of gradient field is of order $\epsilon^{-1/2}$ in dimension two \cite{ammari2007optimal,ammari2005gradient}, and the blowup rate decreases to $\epsilon^{-1/2+\beta}$ for some generic constant $\beta\in\mathbb{R}_+$ in dimensions more than two \cite{li2022gradient}. When the inclusions are balls, the constant $\beta$ was explicitly given in \cite{weinkove2022insulated,dong2021optimal} and converges to $1/2$ as $n \rightarrow \infty$.
For perfect conductors, it was shown that the blowup rate of optimal gradient estimate is of order $\epsilon^{-1/2}$ in dimension two \cite{ yun2009optimal, yun2007estimates},  order $|\epsilon \ln \epsilon|^{-1}$ in dimension three and order $\epsilon^{-1}$ in dimension greater than three \cite{bao2009gradient,bao2010gradient}. It is emphasised that in all of the aforementioned literature, the maximum/comparison principle generally plays a crucial role, from both mathematical and physical perspectives. There are very few works concerning the gradient estimates for wave fields in the frequency regime, where the maximum principle fails. The corresponding study is motivated by the theory of composite materials in photonics and phononics, where one utilises high-contrast building blocks to unlock the potential of materials that manipulate waves; see e.g. \cite{ammari2017mathematical,LLZ} and the references cited therein for more background introduction. In general, the size of the aforementioned building blocks is smaller than the operating wavelength in order to host at least one waveform, and this is known as the sub-wavelength scale or the quasi-static regime. Hence, it is of practical importance to study the wave field concentration between close-to-touching high-contrast material inclusions in the quasi-static regime, which is the focus of the current article.

In fact, Deng, Fang and Liu \cite{deng2022gradient} recently established sharp gradient estimates for the wave field between two nearly-touching high-contrast radial inclusions associated with the 2-D Helmholtz system that arise in transverse electromagnetic scattering. We extend the relevant study to the 3-D Helmholtz system, but associated with the acoustic wave scattering. Though we shall mainly work within the quasi-static regime, namely the size of the material inclusions is smaller than operating wavelength, the size of the inclusions may possess a multi-scale feature with respect to the asymptotic distance parameter $\epsilon$. In fact, in our estimates we cover all the possible scenarios that the size of the inclusions are in different scales in terms of $\epsilon$, which can be used to characterise the curvature effects of the shape of the inclusions on the field concentration; see also \cite{deng2022gradient} for more related discussion about this aspect. Indeed, it is even allowed in \cite{deng2022gradient} that the size of the two inclusions is of different scales. In this paper, we assume that the two radial inclusions of the same size, which allows us to derive much finer gradient estimates. The general gradient estimate consists of two parts: the first part
corresponds to the static field in different scales and in particular the regular-size result recovers the classical result in the literature; and the second part accounts for the frequency effect. It is highly interesting to note that the frequency part may still blow up even if the static part vanishes. This novel phenomenon may have interesting implications to the composite materials. It is known that the sub-wavelength high-contrast inclusions may indue the so-called Minnaert resonance (cf. \cite{ammari2017mathematical,Min,LLZ}). However, we do not consider the resonance effect in our study and hence the blowup of the frequency part is a novel wave phenomenon. Finally, we would like to remark that we shall consider the gradient estimates between resonant sub-wavelength inclusions in a forthcoming paper.

The rest of the paper is organized as follows. In Section 2, we introduce the mathematical setup of our study and discuss the main findings in this paper. In Section 3, we present the mathematical construction of a crucial singular function as well as some auxiliary results that are needed for the gradient estimate. In Section 4, we prove the main theorem and derive the optimal gradient estimate.

\section{MATHEMATICAL SETUP AND MAIN RESULTS}
In this section, we present the mathematical setup of our study and then discuss the main findings in this paper.

Let $\mathbf{c_0} , \mathbf{d_0}$ and $r_c,r_d$ be respectively the centers and radii of two balls $B_j\subset\mathbb{R}^3, j=1,2$, which signify the two material inclusions in our subsequent study. Define $2\epsilon := \mathrm{dist} (B_1,B_2)$. It is assumed that $\epsilon \ll 1$ and $r:=r_c=r_d$.
By rigid motions if necessary, we can assume without loss of generality that
\begin{equation}\label{eq:c0}
\mathbf{c_0} = (r+\epsilon,0,0) \coma \mathbf{d_0} = (-r-\epsilon,0,0) ,
\end{equation}
where $r = r^*\epsilon^\alpha , \alpha \in \bbr$ and $r^*\in\mathbb{R}_+$. The radii are related to the asymptotic distance parameter $\epsilon$ and thus represent different scales of the size of the inclusions. Specially, $\alpha=0$ specifies the regular-size scale of the material inclusions. On the other hand, from a geometric point of view, $\alpha\neq 0$ can be used to characterise the curvature effects of the shape of the inclusion on the gradient estimate; see Fig.~\ref{fig} and Remark~\ref{rem:2.4} in what follows for more discussion. In fact, $\alpha<0$ can characterise the low-curvature scenario whereas $\alpha>0$ can characterise the high-curvature scenario; see also \cite{deng2022gradient} for related discussion about this point.

Next, we consider the scattering of acoustic waves in a uniformly homogeneous medium with two high-contrast balls $B_1 , B_2$ embedded inside.
We denote by $\rho_1 , \kappa_1$ the density and the bulk modulus inside $B_1 \cup B_2$ respectively and $\rho , \kappa$ are the corresponding material parameters for the background medium. We assume in this paper that the contrast between the densities is high whereas between the bulk moduli is not significant. 
The high density contrast between the inclusion and the background medium is one of the key factors to generate the blowup of the gradient field, which is also physically relevant as discussed earlier.
By normalization, we can assume that $\rho = \kappa = 1$ and $\rho_1 \ll 1 , \kappa_1 = \mathcal{O}(1)$.
Then the scattering system is described as the following 3-D Helmholtz system \cite{deng2022gradient}:
\begin{equation}\label{1-1}
\begin{cases}
&\Delta u + \omega^2 u = 0
 \hspace{80pt} \textit{in} \quad \bbr ^3 \backslash \overline{B_1 \cup B_2},\\
&\nabla \cdot \left( \frac{1}{\rho_1}\nabla u \right) + \frac{\omega^2}{\kappa_1} u = 0
\hspace{40pt}\textit{in} \quad B_1\cup B_2, \\
&u \left|_+ = u \right|_-   \;, \; \frac{\partial u}{\partial v}\left|_+ = \frac{1}{\rho_1}\frac{\partial u}{\partial v}\right|_- \hspace{10pt} \textit{on} \quad \partial B_1 \cup \partial B_2 ,\\
&u-u^i \;  \textit{satisfies the Sommerfeld radiation condition}.
\end{cases}
\end{equation}
where $\omega$ signifies the angular frequency of the wave propagation. $u^i$ and $u$ denote the incident and total wave field respectively. $u^i$ is the entire solution of the Helmholtz equation $\Delta u^i + \omega^2 u^i =0$ in $\mathbb{R}^3$. There is a special case that $u^i = e^{\mathrm{i} \omega \mathbf{x}\cdot \theta }$ where $\theta\in\mathbb{S}^2$ signifies the unit impinging direction and $\mathrm{i} :=\sqrt{-1}$. The Sommerfeld radiation condition specifies that
\begin{equation}
\lim_{|\mathbf{x}| \rightarrow \infty} |\mb{x}|\left(\frac{\partial u^s (\mb{x})}{\partial |\mb{x}|}-\mathrm{i}\omega u^s (\mb{x}) \right) = 0,
\end{equation}
where $u^s (\mb{x}) = u(\mb{x})-u^i(\mb{x})$ is the scattered wave and the limit is assumed to hold uniformly in all directions $\mb{x}/|\mb{x}|\in\mathbb{S}^2$.

Throughout the rest of the paper, we assume that
\begin{equation}\label{eq:asum1}
\kappa_1 = \mathcal{O}(1) ,\ \rho_1 \ll 1\ \ \mbox{and}\ \ \omega = \mathcal{O}(\rho_1).
\end{equation}

\begin{theorem}\label{t1}
suppose $\omega \cdot r \ll 1$, and $u^i = u^i_0 +\sum_{j=1}^\infty \omega^ju_j^i \coma j =1,2.\cdots$, where $u_j^i$ is independent of $\omega$. Then for any bounded domain $\Omega$ containing $B_1$ and $B_2$ and when $\alpha <1$, the gradient  of the solution $\nabla u$ to (\ref{1-1}) has the following estimate:
\begin{equation}\label{thm}
\begin{aligned}
\|\nabla u\|_{L^\infty(\Omega\backslash \overline{B_1\cup B_2})} \sim \left.\frac{r^*}{\epsilon^{1-\alpha}\left((1-\alpha)|\ln \epsilon|+\mathcal{O}(1) \right)} \right(& \partial_{\mb{x}_1}u^i(\mb{x}^*) \cdot M + \frac{A \omega^2}{4\pi r^2}  \int_{B_1\cup B_2}u^i(\mb{x})\\
&\left. + \Big(\frac{1}{4\pi r^2} + \frac{Q}{r}\Big) \mathcal{O}(\omega^2)\right) + \mathcal{O}(1).
\end{aligned}
\end{equation}
where $\mb{x}^*$ is some point lying at the line between the centers of $B_1$ and $B_2$. $M$ is defined in  (\ref{M}) in what follows and is a certain finite constant independent of $\alpha , \epsilon ,\omega $ and $u^i$. $A$ is defined  in  (\ref{N})  and has the same order with $Q$ defined in (\ref{34}). The estimate of $\nabla u$ is uniformly bounded for $\alpha \geq 1$.
\end{theorem}

Some remarks on Theorem \ref{t1} are in order.

\begin{remark}
It is noted that the condition $\omega\cdot r\ll 1$ indicates that the size of the inclusion, namely $r$, is smaller than the operating wavelength $2\pi/\omega$. That is, we are working within the quasi-static regime as discussed earlier. We observe that if $\alpha =0$, one has
\begin{equation}\label{eq:me1}
\begin{aligned}
\|\nabla u\|_{L^\infty(\Omega\backslash \overline{B_1\cup B_2})} \sim \frac{r^*}{\epsilon |\ln \epsilon|} (1+o(1))&\left( \partial_{\mb{x}_1}u^i(\mb{x}^*) \cdot M + \frac{A \omega^2}{4\pi r^2}  \int_{B_1\cup B_2}u^i(\mb{x}) \right.\\
&\left. + \Big(\frac{1}{4\pi r^2} + \frac{Q}{r}\Big) \mathcal{O}(\omega^2)\right) + \mathcal{O}(1).
\end{aligned}
\end{equation}
which recovers the blowup rate $(\epsilon |\ln \epsilon|)^{-1}$ of the gradient estimate for the static case \cite{lim2009blow}.
\end{remark}

\begin{remark}
It is pointed out that in \eqref{eq:me1} the static part may vanish, namely $ \partial_{\mb{x}_1}u^i(\mb{x}^*) = 0$, provided the incident field fulfils that $ \sum_{n=0}^\infty q_n (u^i(\mb{c_n}) - u^i(-\mb{c_n})) = 0$, where $\mathbf{c}_n$ and $q_n$ are respectively defined in \eqref{cn}--\eqref{eq:cn} and \eqref{34} in what follows. This is in sharp contrast to the two-dimensional result in \cite{deng2022gradient} where the vanishing of the static part is only related to the incident field at the origin. In fact, we would like to point out that similar phenomenon was observed in \cite{kang2014characterization} for the gradient estimate associated with the Laplace equation. Furthermore, we note that in case $0 < \alpha <1$ and $\partial_{\mb{x}_1}u^i(\mb{x}^*) \neq 0$, which corresponds to the high-curvature scenario, the blowup rate is $((1-\alpha)\epsilon^{1-\alpha}|\ln \epsilon|)^{-1}$, which is smaller than the regular-scale blowup rate $(\epsilon |\ln \epsilon|)^{-1}$ in the static case. It can be also seen that the gradient estimate still holds in the low-curvature case (i.e. $\alpha<0$). The blowup rate is bigger than that in static case. Finally, there is no blowup in the case $\alpha \geq 1$.
\end{remark}

\begin{remark}
We emphasize that the gradient field can still blow up due to the frequency effect even if $\partial_{\mb{x}_1}u^i(\mb{x}^*) = 0$.
In fact, in the case $A\int_{B_1\cup B_2} u^i \neq 0$, which is equivalent to $\int_{B_1\cup B_2} h_0 u^i \neq 0$ (see (\ref{N})) where $h_0$ is defined in (\ref{h0}) and  $0< \alpha <1$ , the blowup occurs if the frequency satisfies
$$C_a \, \epsilon^{\frac{1+\alpha}{2}} < \omega \ll 1,$$
where $C_a = 2 \sqrt{\pi r^* Q/A}$ is a bounded constant.
Furthermore, even if $A\int_{B_1\cup B_2}u^i = 0$, the gradient field may blow up for $0<\alpha<1$ if the frequency satisfies
$$C_b \, \epsilon^{\frac{1+\alpha}{2}}Q^{\frac{1}{2}} < \omega \ll 1 ,$$
where $C_b = 2\sqrt{\pi r^*}$ or in the case $\alpha = 0$ and $\epsilon^{1/2} < \omega \ll 1$.
\end{remark}

\begin{remark}\label{rem:2.4}
It is interesting to note that the frequency effect decreases rapidly as the size of the inclusions increases especially for the low-curvature case, namely $\alpha <0$ (see Figure \ref{fig}) . This means that the field concentration mainly depends on the static effect for large-scale inclusions and the low curvature would weaken the influence of the frequency part.
\end{remark}
\begin{figure}[!htbp]
\begin{center}
\begin{tikzpicture}
\draw[black] (2,2) arc(91.5:88.5:160);
\draw[black] (2,2.5) arc(268.5:271.5:160);
\draw[line width = 1.3pt] (6.1,2.05) -- (6.1,2.45);
\node at (6.3,2.25) {$ \scriptscriptstyle 2\epsilon$};
\node at (5,1.5) {$B_2$};
\node at (5,2.9) {$B_1$};
\end{tikzpicture}
\end{center}
\caption{Schematic illustration of the low-curvature case.}\label{fig}
\end{figure}
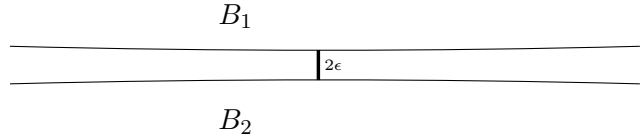

\section{MATHEMATICAL CONSTRUCTION}
In order to prove the main result in Theorem \ref{t1}, we next introduce an auxiliary system and present several auxiliary lemmas, whose proofs are deferred to the subsequent subsections. We first introduce the following Dirichlet system:
\begin{equation}\label{Q-1}
\begin{cases}
\Delta u + \omega^2 u = 0  \quad in \quad \bbr^3 \backslash \overline{B_1 \cup B_2}, \\
u = \lambda_1 + \mathcal{O}(\omega^2) \quad on \quad \partial B_1, \\
u = \lambda_2 + \mathcal{O}(\omega^2) \quad on \quad \partial B_2 ,\\
u - u^i \;  \textit{satisfies the Sommerfeld radiation condition} ,
\end{cases}
\end{equation}
where the constants $\lambda_j$ are determined by
\begin{equation*}
\int_{\partial B_j} \partial_\nu u|_+ = \mathcal{O}(\omega^2) \coma j=1,2
\end{equation*}
and they are unique up to $\mathcal{O}(\omega^2)$.

We have the following result:
\begin{lemma}\label{lemma1}
Let $v$ be the solution of the system (\ref{1-1}) in $\bbr^3\backslash \overline{B_1 \cup B_2}$. Suppose $u$ is solution to (\ref{Q-1}). Then there holds
\begin{equation*}
v=u+C\omega+\mathcal{O}(\omega^2),
\end{equation*}
where $C$ is a constant which does not depend on $\epsilon$ and $\omega$.
\end{lemma}

Let $R(\mb{x})$ signify the reflection with respect to $\partial B_2$:
\begin{equation*}
R(\mb{x}) = \frac{r^2   (\mb{x}-\mb{d_0})}{|\mb{x}-\mb{d_0}|^2}+ \mb{d_0} \coma \mb{x} \in \bbr^3\backslash \overline{B_2}.
\end{equation*}
Set $\mb{c_0}=(c_0, 0, 0)$ with $c_0=r+\epsilon$ to be the center of $B_1$ (cf. \eqref{eq:c0}). Define $\mb{c}_n=(c_n, 0, 0)$ with $c_n$ defined by the following recursive relation:
\begin{equation}\label{cn}
c_{n+1} = - \frac{r^2}{r+\epsilon+c_n} +r+\epsilon,\quad n=0, 1, 2, \ldots.
\end{equation}
One can verify that
\begin{equation}\label{eq:cn}
\mb{c_{n+1}} =(c_{n+1},0,0) =  -R(\mb{c_n}),\quad n=0,1,2,\ldots,
\end{equation}
and the sequence $\{\mb{c_n}\}_{n\in\mathbb{N}}$ converges to $\mb{c}=(c,0,0)=(\sqrt{\epsilon^2 +2r\epsilon},0,0)$ as $n\rightarrow\infty$. It is noted that both $\mathbf{c}_n$, $n\in\mathbb{N}$, and $\mathbf{c}$ stay on the line connecting $(\epsilon,0,0)$ and $(c_0,0,0)$ inside the ball $B_1$, no matter what value of the radius $r$ takes.

Next we let $h_0$ be a singular function defined by
\begin{equation}\label{h0}
h_0(\mb{x}) = -\frac{1}{4\pi Q}\sum_{n=0}^\infty q_n \left( \frac{1}{|\mb{x}-\mb{c_n}|} - \frac{1}{|\mb{x}+\mb{c_n}|} \right),
\end{equation}
where
\begin{align}
\rho_m &=
\begin{cases}\label{34}
1 \coma  m=0, \\
\frac{r}{|c_0+c_{m-1}|} \coma m \geq 1,
\end{cases}\ \mbox{and}\ \
q_n = \prod_{m=0}^n \rho_m  \coma
Q = \sum_{n=0}^\infty q_n.
\end{align}
It can be verified by straightforward calculations that $h_0$ is the solution of the following system \cite{yun2007estimates,lim2009blow}:
\begin{equation}\label{3-2}
\begin{cases}
\Delta h_0 = 0 \qquad in \quad \bbr^3 \backslash \overline{B_1 \cup B_2}, \\
h = C_j \hspace{28pt} on \quad \partial B_j , \\
\int_{\partial B_j} \partial_{\nu} h_0 |_+ = (-1)^{j+1} \quad j=1,2, \\
h_0 (\mathbf{x}) = \mathcal{O}(|\mb{x}|^{-2})  \quad as \quad |\mathbf{x}| \rightarrow \infty,
\end{cases}
\end{equation}
where
$
C_j := (-1)^j \dfrac{1}{4\pi r Q} \coma j =1,2 .
$

In what follows, we shall decompose the solution $u(\mb{x})$ to the system (\ref{Q-1}) into two parts and define
\begin{equation}\label{b}
b(\mb{x}) := u(\mb{x}) - u_\omega (\mb{x}) = u(\mb{x})  -a h_\omega (\mb{x}),
\end{equation}
where $a = (\lambda_1 -\lambda_2)/(C_1-C_2)$ and $C_j,\lambda_j$ are defined in (\ref{Q-1}) and (\ref{3-2}). We have the following result:
\begin{lemma}\label{lemma2}
Let $b(\mb{x})$ be defined in (\ref{b}). Then for any bounded domain $\Omega \subset \bbr^3$ containing the balls $B_1$ and $B_2$ , it holds that
\begin{equation}
\|\nabla b\|_{L^\infty(\Omega \backslash \overline{B_1 \cup B_2})} \leqslant C  +\mathcal{O}(\omega^2),
\end{equation}
where C is some constant independent of $\epsilon \coma \omega $ and $\alpha$ .
\end{lemma}


\subsection{Layer potentials}
We shall make use layer-potential techniques to prove the results in the previous subsection. To that end, we introduce some necessary notations and results on the layer potential operators. Let $\Gamma_\omega$ be the fundamental solution of the 3-D PDO $\Delta+\omega^2$:
\begin{equation}
\Gamma_\omega (\mb{x})= -\frac{e^{\mathrm{i}\omega |\mb{x}|}}{4\pi|\mb{x}|},\ \ \mb{x}\neq \mathbf{0}.
\end{equation}
We mention that $\Gamma_0 (\mb{x}) = -\dfrac{1}{4\pi|\mb{x}|}$ if $\omega =0$. Let $B \subset \bbr ^3$ be a bounded $C^2$ domain, the single- and double-layer potential operators are bounded from $L^2(\partial B)$ into $H^1(\bbr^2 /\partial B)$ and given by
\begin{align}
S^\omega_B [\psi] (\mb{x})  &= \int_{\partial B} \Gamma_\omega (\mb{x}-\mb{y}) \psi(\mb{y}) \dx{s_\mb{y}} \coma \mb{x} \in \bbr^3/\partial B, \\
D^\omega_B [\psi] (\mb{x}) &= \int_{\partial B} \frac{\partial \Gamma_\omega (\mb{x}-\mb{y})}{\partial \nu_\mb{y}} \psi(\mb{y}) \dx{s_\mb{y}} \coma \mb{x} \in \bbr^3/\partial B.
\end{align}
Here $\nu_{\mb{x}}\in\mathbb{S}^2$ signifies the exterior unit normal vector to the boundary of the concerned domain $B$ at $\mb{x}$. The Neumann-Poincar\'e operators $K_B^\omega , (K_B^\omega)^*$ are bounded from $L^2(\partial B)$ into $L^2(\partial B)$ and given by
\begin{align}
K_B^\omega [\psi] (\mb{x}) &= \int_{\partial B} \frac{\partial \Gamma_\omega (\mb{x}-\mb{y})}{\partial \nu_\mb{y}} \psi(\mb{y}) \dx{s_\mb{y}} \coma \mb{x} \in \partial B, \\
(K_B^\omega)^* [\psi] (\mb{x})  &= \int_{\partial B} \frac{\partial \Gamma_\omega (\mb{x}-\mb{y})}{\partial \nu_\mb{x}} \psi(\mb{y}) \dx{s_\mb{y}} \coma \mb{x} \in \partial B.
\end{align}
Furthermore, it is known that the single-layer potential is continuous across $\partial B$ and the normal derivative satisfies the trace formula
\begin{equation}
\left. \frac{\partial}{\partial \nu}S^\omega_B [\psi] (\mb{x})\right|_{\pm} = \left(\pm\frac{I}{2} + (K_B^\omega)^*\right) [\psi](\mb{x}) \coma \mb{x}\in \partial B,
\end{equation}
where the subscripts $\pm$ indicate the limits from outside and inside of the domain $B$ respectively. The double-layer potential satisfies the trace formula
\begin{equation}
\left. D^\omega_B [\psi] (\mb{x}) \right|_{\pm} = \left(\mp \frac{I}{2} +K_B^\omega \right) [\psi] (\mb{x}) \coma \mb{x}\in \partial B.
\end{equation}
It is pointed out that the above formulas still  hold when $\omega =0$ and the kernel function shall be replaced by $\Gamma_0(\mb{x})$.

Next we recall some useful asymptotic expansions for the layer potentials \cite{ammari2017mathematical,DLZ1,DLZ2}. The fundamental solution has the following asymptotic expansion:
\begin{equation}
\Gamma_\omega (\mb{x}) = \Gamma_0 (\mb{x}) -\sum_{j=1}^\infty \frac{\mathrm{i}}{4\pi}\frac{(\mathrm{i}|\mb{x}|)^{j-1}}{j!} \omega^j.
\end{equation}
then we have the following asymptotic expansion:
\begin{equation}\label{2-1}
S^\omega_B [\psi] (\mb{x}) = S_B^0 [\psi] (\mb{x}) +\sum_{j=1}^\infty \omega^j S_B^j [\psi] (\mb{x})\coma \mb{x}\in \partial B,
\end{equation}
where
\begin{equation*}
S_B^j [\psi] (\mb{x}) = -\frac{\mathrm{i}^j}{4\pi j!} \int_{\partial B} |\mb{x}-\mb{y}|^{j-1}\psi (\mb{y})\dx{s_\mb{y}} .
\end{equation*}

Similarly, the Neumann-Poincar\'e operator $(K_B^\omega)^*$ has the following asymptotic expansion:
\begin{equation}\label{2-2}
(K_B^\omega)^* [\psi] (\mb{x}) = (K_B^0)^* [\psi] (\mb{x}) + \sum_{j=1}^\infty \omega^j K_B^j [\psi] (\mb{x}) \coma \mb{x}\in \partial B,
\end{equation}
where
\begin{equation*}
K_B^j [\psi] (\mb{x}) = -\frac{\mathrm{\mathrm{i}}^j (j-1)}{4\pi j!}\int_{\partial B} |\mb{x}-\mb{y}|^{j-3}(\mb{x}-\mb{y}) \cdot \nu_{\mb{x}} \psi(\mb{y}) \dx{s_\mb{y}}.
\end{equation*}

By the operator theory, we can conclude the following result about the boundedness on the asymptotic expansion:
\begin{lemma}
The norm $\|S_D^j\|_{\mathcal{L}(L^2(\partial B),H^1(\partial B))}$ is uniformly bounded with respect to $j$. Moreover, the series (\ref{2-1}) is convergent in $\mathcal{L}(L^2(\partial B),H^1(\partial B))$.
\end{lemma}
\begin{lemma}
The norm $\|K_D^j\|_{\mathcal{L}(L^2(\partial B))}$ is uniformly bounded with respect to $j$. Moreover, the series (\ref{2-2}) is convergent in $\mathcal{L}(L^2(\partial B))$.
\end{lemma}
In this paper, we let $h_\omega (x)$ be defined as the following function:
\begin{equation}
h_\omega (\mb{x}) := -\frac{1}{4\pi Q}\sum_{n=0}^\infty q_n \left( \frac{e^{\mathrm{i}\omega |\mb{x}-\mb{c_n}|}}{|\mb{x}-\mb{c_n}|} - \frac{e^{\mathrm{i}\omega |\mb{x}+\mb{c_n}|}}{|\mb{x}+\mb{c_n}|} \right) = h_0(\mb{x}) + g_\omega (\mb{x}).
\end{equation}

\subsection{Properties of $\{\mathbf{c}_n\}$, $q_n$ and $Q$}

In this subsection, we investigate the properties of $\{\mathbf{c}_n\}$ as well as the concerned quantities $q_n$ and $Q$. It is  mentioned that the related results has been well settled in \cite{kang2014characterization} if the radius $r$ is equal to $1$ that does not depend on asymptotic parameter $\epsilon$ . We need to consider the case where $r$ is in different scales. For the rest of this paper, we use the following notations for $\alpha <1$:
\begin{align}\label{c-1}
\delta = \epsilon^{1-\alpha}  \coma N := \left[ \epsilon^{\frac{\alpha -1}{2}} \right] \coma p_n = \frac{c_n}{r},
\end{align}
where $[\,\cdot\,]$ is the Gaussian bracket. Straightforward calculation shows that $\{p_n\}$ satisfies the formula
\begin{equation}\label{pn}
p_n = p\left(\frac{2}{A^{n+1}-1}+1\right),
\end{equation}
where
\begin{equation*}
A = \frac{1+\delta+p}{1+\delta-p} \coma p = \sqrt{\delta^2 +2\delta}
\end{equation*}
and $p$ is the convergent point of $\{p_n\}$.

We prove the following lemma.
\begin{lemma}\label{le1}
Let $\delta , N , p_n$ be defined as (\ref{c-1}). Then it holds that
\begin{equation}
p_n  = \frac{1}{n+1} + \mathcal{O}(\sqrt{\delta}) \coma q_n = \frac{1}{n+1}\left( 1 + \mathcal{O}(\sqrt{\delta})\right) ,
\end{equation}
for $n \leqslant N$ and $p_n = \mathcal{O}(\sqrt{\delta})$ for $n > N$, and there is a constant $C$ independent of $\delta$ such that
\begin{equation}
\sum_{n=0}^\infty q_n = |\log \sqrt{\delta}| + C.
\end{equation}
\end{lemma}
\begin{proof}
When $n \leqslant N$, it holds that
\begin{equation*}
A = 1+2\sqrt{2}\delta^{\frac{1}{2}}+4\delta + \mathcal{O}(\delta^{\frac{3}{2}})  \coma A^{n+1} = 1+2\sqrt{2}(n+1)\delta^{\frac{1}{2}} +(n+1)^2 \mathcal{O}(\delta).
\end{equation*}
It follows from the formula (\ref{pn}) that
\begin{equation*}
p_n = \frac{1}{n+1} + \mathcal{O}(\sqrt{\delta}).
\end{equation*}
Since $p_n$ decays to $p$, one can show that for $n>N$,
\begin{equation*}
\sqrt{2\delta} \leqslant p_n \leqslant C\sqrt{\delta},
\end{equation*}
where $C$ is some constant independent of $\delta$.

We have from (\ref{34}) that for $m>1$,
\begin{align*}
\rho_m &= \frac{r}{|c_0+c_{m-1}|} = \frac{1}{1+\delta+p_{m-1}} = \frac{m}{m+1}(1+ \mathcal{O}(\sqrt{\delta})), \\
q_n &= \prod_{m=0}^n \rho_m = \frac{1}{n+1} (1+ \mathcal{O}(\sqrt{\delta})) \coma n \leqslant N.
\end{align*}
The recurrence formula (\ref{cn}) implies that $\rho_m = 1/(1+\delta +p_{m-1}) = 1+\delta-p_m$. Then we have that $q_n \leqslant (1+\delta -p)q_{n-1}$ holds for all $n>1$ and furthermore
\begin{equation*}
\sum_{n=N}^\infty q_n \leqslant \sum_{n=N}^\infty q_N (1+\delta-p)^{n-N} \leqslant \frac{1}{N+1}(1+\mathcal{O}(\sqrt{\delta})\sum_{n=N}^\infty (1+\delta-p)^{n-N} \leqslant C,
\end{equation*}
and
\begin{equation*}
\left|\sum_{n=0}^\infty q_n- \sum_{n=0}^N\frac{1}{n+1} \right| \leqslant CN\sqrt{\delta} + \sum_{n=N}^\infty q_n \leqslant C.
\end{equation*}
Hence we obtain
\begin{equation*}
Q = \sum_{n=0}^{N-1} q_n + \sum_{n=N}^\infty q_n = |\log \sqrt{\delta}| +C.
\end{equation*}

The proof is complete.
\end{proof}

On the other hand, when $\alpha \geq 1$, we redefine the notations
$
\delta = \epsilon^{\alpha -1} , p_n = \frac{c_n}{\epsilon}.
$
Then the recurrence formula is
\begin{equation}
p_n = p \left(\frac{2}{A^{n+1}-1} +1 \right) \; where \; A = \frac{1+\delta+p}{1+\delta-p} \coma \delta = \epsilon^{\alpha -1},
\end{equation}
and the convergent point is $ p = \sqrt{1+2\delta}$.

We have the following lemma.

\begin{lemma}\label{le2}
For all $\alpha \geq 1$, there are constants $C_1 , C_2$ only dependent of $r^*$ such that
\begin{equation}
C_1 \leqslant Q \leqslant C_2.
\end{equation}
\end{lemma}
\begin{proof}
Since $\epsilon \leqslant c_n \leqslant r +\epsilon$ , we see that
\begin{equation*}
\frac{r}{2r+2\epsilon} \leqslant \rho_n \leqslant \frac{r}{2\epsilon +r} \coma \forall n \in \bbn.
\end{equation*}
When $\alpha \geq 1$ , one sees that $ 0 < \epsilon^{\alpha-1} \leqslant 1$. Hence, we have that
\begin{equation*}
\frac{r^* \epsilon^{\alpha -1}}{2 + 2r^*\epsilon^{\alpha -1}} \leqslant \rho_n \leqslant \frac{r^* \epsilon^{\alpha -1}}{2 + r^*\epsilon^{\alpha -1}},
\end{equation*}
which further yields that
\begin{equation*}
0 < 1 + \frac{r^* \epsilon^{\alpha -1}}{2+r^*\epsilon^{\alpha-1}} \leqslant Q \leqslant 1+ \frac{1}{2}r^* \epsilon^{\alpha -1} \leqslant 1+ \frac{1}{2}r^*.
\end{equation*}

The proof is complete.
\end{proof}

\section{QUANTITATIVE APPROXIMATIONS}
In this section, we give the proofs of the key Lemmas \ref{lemma1} and \ref{lemma2} and estimate the key quantities  in the formula (\ref{b}).

\subsection{Proof of Lemma \ref{lemma1}}

By applying the layer-potential technique, one can represent the solution of (\ref{1-1}) by
\begin{equation}
u(\mb{x}) =
\begin{cases}
u^i(\mb{x})+ S_{B_c}^\omega [\varphi_1](\mb{x}) \coma  in  \; \bbr^3 \backslash \overline{B_c}, \\
S^{k_c}_{B_c} [\varphi_2](\mb{x}) \coma \hspace{38pt}  in  \; B_c,
\end{cases}
\end{equation}
for some surface potentials $\varphi_1,\varphi_2 \in L^2(\partial B_c)$ where $B_c =B_1 \cup B_2$ and $k_c =\omega \sqrt{\dfrac{\rho_1}{\kappa_1}}$.
Using the jump relations for the single layer potentials, one can show that the transmission problem is equivalent to the following boundary integral equation:
\begin{equation}\label{au}
A(\omega,\rho_1) [\varphi] = U  \quad on \quad \partial{B_c},
\end{equation}
where
\begin{align*}
A(\omega,\rho_1) = \begin{bmatrix}
-S^\omega_{B_c} &  S^{k_c}_{B_c} \\
-\rho_1 (\frac{I}{2}+(K^\omega_{B_c})^*) & -\frac{I}{2} +(K^{k_c}_{B_c})^*
\end{bmatrix}
\coma
[\varphi] = \begin{bmatrix}
\varphi_1 \\ \varphi_2
\end{bmatrix}
\coma
U = \begin{bmatrix}
u^i \\ \rho_1 \partial_v u^i
\end{bmatrix}.
\end{align*}
It is clear that $A(\omega ,\rho_1)$ is a bounded linear operator form $\mathcal{H}_1$ to $\mathcal{H}_2$ where $\mathcal{H}_1 = L^2(\partial B_c) \times L^2(\partial B_c)$ and  $\mathcal{H}_2 = H^1(\partial B_c) \times L^2(\partial B_c)$. In order to prove Lemma \ref{lemma1}, we note that
\begin{equation}\label{expan}
S^0_{B_c} :=
\begin{bmatrix}
S^0_{B_1}|_{\partial _{B_1}} & S^0_{B_2}|_{\partial _{B_1}} \\
S^0_{B_1}|_{\partial _{B_2}} & S^0_{B_2}|_{\partial _{B_2}}
\end{bmatrix}
\quad \coma \quad
(K^0_{B_c})^* :=
\begin{bmatrix}
(K^0_{B_1})^* & \partial_{\nu_1} S^0_{B_2} \\
\partial_{\nu_2} S^0_{B_1} & (K^0_{B_2})^*
\end{bmatrix},
\end{equation}
where $\nu_1 , \nu_2$ are the exterior unit normal vectors to $\partial B_1$ and $\partial B_2$, respectively.

\begin{proof}[Proof of Lemma \ref{lemma1}:]
By the asymptotic formulas (\ref{2-1}) and (\ref{2-2}) for layer-potentials, one can show that
\begin{equation}
\begin{aligned}
S^\omega_{B_c} [\varphi] &= S^0_{B_c} [\varphi] -\frac{\mathrm{i}\omega}{4\pi} \int_{\partial B{_c}} \varphi +  \mathcal{O}(\omega^2), \\
(K^\omega_{B_c})^* [\varphi] &= (K^0_{B_c})^* [\varphi] + \mathcal{O}(\omega^2).
\end{aligned}
\end{equation}
It then follows from (\ref{au})  that
\begin{equation}
\begin{aligned}
\left( -\frac{I}{2} + (K^0_{B_c})^* \right) [\varphi_2] &=\rho_1 \left( \frac{I}{2} + (K^0_{B_c})^* \right) [\varphi_1]+\rho_1 \partial_v u_0^i+\mathcal{O}(\omega^2),\\
-S^0_{B_c} [\varphi_1] + S^0_{B_c} [\varphi_2] &=u^i +
\frac{\mathrm{i}k_c}{4\pi}\int_{\partial B_c} \varphi_2 - \frac{\mathrm{i}\omega}{4\pi}\int_{\partial B{_c}}\varphi_1
+\mathcal{O}(\omega^2).
\end{aligned}
\end{equation}
Then straightforward asymptotic analysis shows that
\begin{equation}\label{eq:tmpapp01}
 \left( -\frac{I}{2} + (K^0_{B_c})^* \right) [\varphi_2]=\mathcal{O}(\omega)\coma  \int_{\partial B_j} \varphi_1 = \mathcal{O}(\omega^2), j=1, 2.
\end{equation}
Thus one has
\begin{equation}
S^0_{B_c} [\varphi_2]  = \lambda_1 \chi (\partial {B_1}) + \lambda_2 \chi (\partial {B_2})+C_1\omega+\rho_1S^0_{B_c}[\partial_v u_0^i]+\mathcal{O}(\omega^2),
\end{equation}
where $C_1$, $\lambda_1 , \lambda_2$ are some generic constants. Note that $\nabla S^0_{B_c}[\partial_v u_0^i]$ is uniformly bounded with respect to $\epsilon$.
Hence,
\begin{equation}
\begin{aligned}
u &= u^i +S^\omega_{B_c} [\varphi_1] = u_0^i + \omega u_1^i +S^0_{B_c}[\varphi_1] -\frac{\mathrm{i}\omega}{4\pi} \int_{\partial B_c} \varphi_1 + \mathcal{O}(\omega^2) \\
&= \begin{cases}
 \lambda_1 + C_1\omega + \mathcal{O}(\omega^2)  \quad on \quad \partial B_1 , \\
 \lambda_2 + C_1\omega + \mathcal{O}(\omega^2)  \quad on \quad \partial B_2 .
\end{cases}
\end{aligned}
\end{equation}
Moreover, it holds that
\begin{equation}
\begin{aligned}
 \int_{\partial B_j} \left. \partial_\nu u\right|_+ &= \int_{\partial B_j} \left.\partial_\nu u^i \right|_+ + \left.\partial_{\nu_j} S^\omega_{B_c}[\varphi_1]\right|_+ \\
&= \int_{\partial B_j} \left. \partial_{\nu_j} u^i_0 \right|_+ +\omega \partial_{\nu_j} u_1^i + \Big(\frac{I}{2}+(K^0_{B_j})^*\Big)[\varphi_{B_j}] + \partial_{\nu_j}S^0_{B_l} [\varphi_{B_l}] + \mathcal{O}(\omega^2)\\
&= \mathcal{O}(\omega^2) \coma j =1,2
\end{aligned}
\end{equation}
where $l=3-j \coma \varphi_1 = (\varphi_{B_1} , \varphi_{B_2}) \in L^2(\partial B_c)$ and we have used the third formula in (\ref{eq:tmpapp01}). This readily completes the proof.
\end{proof}

\subsection{Estimate of $u_\omega (\mb{x})$}
In this subsection, we need to estimate the main quantities on the right side of the formula (\ref{b}). First, we calculate the upper bound and lower bound of $\nabla h_\omega$ in the case $\alpha <1$. For any bounded domain $\Omega$ contains $B_1$ and $B_2$, one can show that
\begin{equation}\label{upp}
\begin{aligned}
\|\nabla h_\omega (\mb{x})\|_{L^\infty (\Omega \backslash\overline{B_1 \cup B_2})} \leqslant & \left\|\frac{1}{4\pi Q}\sum_{n=0}^\infty q_n \left(\frac{\mb{x}-\mb{c_n}}{|\mb{x}-\mb{c_n}|^3} -\frac{\mb{x}+\mb{c_n}}{|\mb{x}+\mb{c_n}|^3}\right)\right\|_{L^\infty (\Omega\backslash\overline{B_1 \cup B_2})} \\
&+\left\|\frac{\omega^2}{8\pi  Q}\sum_{n=0}^\infty q_n \left(\frac{\mb{x}-\mb{c_n}}{|\mb{x}-\mb{c_n}|}-\frac{\mb{x}+\mb{c_n}}{|\mb{x}+\mb{c_n}|}\right) \right\|_{L^\infty (\Omega\backslash\overline{B_1 \cup B_2})} + \mathcal{O}(\omega^2) \\
\leqslant&  \left\|\frac{1}{2\pi Q}\sum_{n=0}^\infty q_n \frac{\mb{x}-\mb{c_n}}{|\mb{x}-\mb{c_n}|^3} \right\|_{L^\infty (\Omega\backslash\overline{B_1 \cup B_2})} +\mathcal{O}(\omega^2) \\
\leqslant&\frac{1}{2\pi Q}\sum_{n=0}^\infty \frac{q_n}{|\epsilon-c_n|^2}  +\mathcal{O}(\omega^2) \\
\leqslant& \frac{C}{2\pi r^*Q}\frac{1}{\epsilon^{1+\alpha}} +\mathcal{O}(\omega^2),
\end{aligned}
\end{equation}
where $C$ is some constant and we have used the quasi-static ansatz ($\omega \cdot r \ll 1$) as well as the result
\begin{align*}
\sum_{n=0}^\infty \frac{q_n}{|\delta -p_n|^2} &= \sum_{n=0}^N \frac{q_n}{|\delta -p_n|^2} +\sum_{N+1}^\infty \frac{q_n}{|\delta -p_n|^2} \sim \mathcal{O}(\delta^{-1}).
\end{align*}
On the other hand,
\begin{equation}
\begin{aligned}
|h_\omega (\epsilon) - h_\omega (-\epsilon)|& = \frac{1}{2\pi Q}\sum_{n=0}^\infty q_n (\frac{1}{|\epsilon - c_n|} -\frac{1}{|\epsilon + c_n|})   \\
&+ \frac{\omega^2}{4\pi  Q}\sum_{n=0}^\infty q_n(|\epsilon -c_n|-|\epsilon +c_n|)  + \cdots\\
&\sim \frac{C}{2\pi r^* Q}\frac{1}{\epsilon^{\alpha}} +\mathcal{O}(\omega^2\epsilon),
\end{aligned}
\end{equation}
where $C$ is some constant and we have used the fact that $c_n$ is bigger than $\epsilon$ for all $n$. The mean value theorem shows that the lower bound satisfies
\begin{equation}\label{low}
\|\nabla h_\omega (\mb{x})\|_{L^\infty (\Omega\backslash\overline{B_1 \cup B_2})}  \geq \frac{C}{2\pi r^* Q}\frac{1}{\epsilon^{1+\alpha}} +\mathcal{O}(\omega^2).
\end{equation}
The gradient estimate for $h_\omega (x)$  is
\begin{equation}
\|\nabla h_\omega (\mb{x})\|_{L^\infty (\Omega\backslash\overline{B_1 \cup B_2})}  \sim\frac{C}{2\pi r^*Q}\frac{1}{\epsilon^{1+\alpha}} + \mathcal{O}(\omega^2).
\end{equation}
Moreover, one has
\begin{equation}\label{lamda}
\begin{aligned}
\lambda_1 - \lambda_2 &= \int_{\partial B_1} \partial_\nu h_0 \, u + \int_{\partial B_2}\partial_\nu h_0 \, u + \mathcal{O}(\omega^2) \\
&=\int_{\partial B_1} \partial_\nu h_0 \, (u-u^i) + \int_{\partial B_2}\partial_\nu h_0 \, (u-u^i) + \int_{\partial B_1 \cup \partial B_2} \partial_\nu h_0 \, u^i + \mathcal{O}(\omega^2)  \\
&= \frac{1}{4\pi rQ}\mathcal{O}(\omega^2) + \int_{\partial B_1 \cup \partial B_2} \partial_\nu h_0 u^i - \int_{\partial B_1 \cup \partial B_2} \partial_\nu u^i h_0 + \mathcal{O}(\omega^2)\\
&= \int_{ B_1 \cup B_2} \Delta h_0 u^i -\int_{ B_1 \cup B_2}  h_0 \Delta u^i +  \frac{1}{4\pi rQ}\mathcal{O}(\omega^2) + \mathcal{O}(\omega^2)\\
&= \frac{1}{Q} \sum_{n=0}^\infty q_n (u^i(\mb{c_n}) - u^i(-\mb{c_n})) + \omega^2\int_{ B_1 \cup B_2}  h_0 \, u^i  + \frac{1}{4\pi rQ}\mathcal{O}(\omega^2) + \mathcal{O}(\omega^2),
\end{aligned}
\end{equation}
where we have used the boundary integral of $u(x)$ and the boundary values of $h_0$ on $\partial B_1 \cup \partial B_2$.

Since $h_0$ on the boundary has an explicit formula, it is clear that
\begin{equation}\label{cc}
|C_1 - C_2| = \frac{1}{2\pi rQ} = \frac{1}{2\pi r^* \epsilon^\alpha Q}.
\end{equation}

\subsection{Estimate of $b(\mb{x})$}
From the definition (\ref{b}), one can find that $b(\mathbf{x})$ satisfies the following system:
\begin{equation}
\begin{cases}
\Delta b+\omega^2 b=0 \hspace{100pt} \text { in } \quad \mathbb{R}^3 \backslash \overline{B_1 \cup B_2}, \\ b=\left(\lambda_2 C_1-\lambda_1 C_2\right) /\left(C_1-C_2\right)  \hspace{17pt} \text { on } \quad \partial B_1 \cup \partial B_2 ,\\ \left(b-u^i\right)(\mathbf{x}) \text { satisfies the Sommerfeld radiation condition. }
\end{cases}
\end{equation}

Next, we prove $\nabla b$ is uniformly bounded with respect to $\epsilon$. Note that there is no potential difference on $\partial B_1$ and $\partial B_2$.
By implementing layer potential techniques, $b(\mathbf{x})$ can be represented by
$$
b(\mathbf{x})=u^i(\mathbf{x})+\mathcal{S}_{B_1}^\omega\left[\phi_1\right](\mathbf{x})+\mathcal{S}_{B_2}^\omega\left[\phi_2\right](\mathbf{x}),
$$
where $(\phi_1, \phi_2)\in L^2\left(\partial B_1\right) \times \in L^2\left(\partial B_2\right)$ satisfy the boundary condition
\begin{equation*}
u^i(\mathbf{x})+\mathcal{S}_{B_1}^\omega\left[\phi_1\right](\mathbf{x})+\mathcal{S}_{B_2}^\omega\left[\phi_2\right](\mathbf{x}) = \tilde{C} \quad  \text { on } \quad \partial B_1 \cup \partial B_2.
\end{equation*}
where $\tilde{C} = \left(\lambda_2 C_1-\lambda_1 C_2\right) /\left(C_1-C_2\right)$,

By the expansion formula \eqref{2-1}, one can see that the first-order term of $\omega$ is constant outside $B_1$ and $B_2$. Then it is clear in a bounded set $\Omega$ that
\begin{equation}
\nabla b(\mathbf{x})=\nabla u^i(\mathbf{x})+\nabla b_{0}+\mathcal{O}(\omega^{2}),
\end{equation}
where $b_{0}=\mathcal{S}_{B_1}^{0}\left[\phi_1\right](\mathbf{x})+\mathcal{S}_{B_2}^{0}\left[\phi_2\right](\mathbf{x}$). It is sufficient to estimate $\nabla b_{0}$ from the above equation. Note that $b_0(\mathbf{x})$ is a harmonic function satisfying the following system:
\begin{equation}
\begin{cases}\Delta b_0=0 & \text { in } \quad \mathbb{R}^3 \backslash \overline{B_1 \cup B_2}, \\ b_0=\tilde{C}-u^i & \text { on } \quad \partial B_1 \cup \partial B_2, \\ b_0(\mathbf{x})=\mathcal{O}\left(|\mathbf{x}|^{-1}\right). & \end{cases}
\end{equation}
\begin{proof}[Proof of Lemma \ref{lemma2}:]
Since $\nabla b_0$ is harmonic in $\mathbb{R}^3 \backslash \overline{B_1 \cup B_2}$, and $\nabla b_0 =\mathcal{O}(|\mb{x}|^{-2})$, it can be seen that the function $\left|\nabla b_0\right|_{l^{\infty}}$ in $\mathbb{R}^3 \backslash\left(B_1 \cup B_2\right)$ attains its maximum on the boundary $\partial B_1 \cup \partial B_2$.
Thus it is enough to show that $\nabla b_0 $ is bounded on the two points $\zeta_1=\left(\epsilon, 0, 0\right)$ and $-\zeta_1=\left(-\epsilon, 0, 0\right)$. By direct calculations, we have
$$
\begin{aligned}
\nabla b_0\left(\zeta_1\right) &=\nu\left(\zeta_1\right) \cdot \nabla b_0\left(\zeta_1\right) \nu\left(\zeta_1\right)+\partial_T b_0\left(\zeta_1\right) T\left(\zeta_1\right) \\
&=\partial_{\mathbf{x}_1} b_0\left(\zeta_1\right)(-1,0,0)-\partial_T u^i\left(\zeta_1\right)T\left(\zeta_1\right)  \\
&=(-1,0,0)  \left(\frac{u^{i}\left(-\zeta_1\right)-u^{i}\left(\zeta_1\right)}{2\epsilon} +O(\epsilon) \right) -\partial_T u^i\left(\zeta_1\right)T\left(\zeta_1\right)  \\
&= -(-1,0,0)\partial_{x_1}u^i(\zeta_1) -\partial_T u^i\left(\zeta_1\right)T\left(\zeta_1\right) + \mathcal{O}(\epsilon) \\
&=-\nabla u^i\left(\zeta_1\right) + \mathcal{O}(\epsilon),
\end{aligned}
$$
where $T$ is the unit tangential vector and $\nu$ is the unit normal vector. Here we have used the fact that $b_0$ takes the identical value on $\partial B_1$ and $\partial B_2$. $\nabla b_0(\zeta)$ is uniformly bounded since $\nabla u^i (\zeta)$ is uniformly bounded with respect to $\epsilon$.
Similarly, $\nabla b_0\left(-\zeta_1\right)=-\nabla u^i\left(-\zeta_1\right) + \mathcal{O}(\epsilon)$.  Thus $\nabla b_0\left(\zeta_1\right)$ and $\nabla b_0\left(-\zeta_1\right)$ are uniformly bounded with respect to $\epsilon$.

The proof is complete.
\end{proof}

\subsection{Proof of Theorem \ref{t1}}

We are in a position to prove the main Theorem \ref{t1}. First, Lemmas \ref{le1} and \ref{le2} yield that
\begin{equation}
Q =
\begin{cases}
\frac{1-\alpha}{2}|\ln \epsilon| + \mathcal{O}(1) \coma \alpha <1, \\
\mathcal{O}(1) \coma \hspace{55pt} \alpha \geq 1.
\end{cases}
\end{equation}

If $\alpha <1$, the asymptotic results (\ref{upp}) and (\ref{low}) show that one needs to estimate the main part $\nabla h_0$.
By (\ref{lamda}), it is clear that
\begin{equation}\label{N}
\begin{aligned}
\lambda_1 - \lambda_2
=&\frac{1}{Q} \sum_{n=0}^\infty q_n  \nabla u^i(\mb{x_n}) \cdot 2\mb{c_n} +
\frac{\omega^2}{4\pi Q}\sum_{n=0}^\infty q_n \int_{B_1\cup B_2}u^i \left(\frac{1}{|\mb{x}+\mb{c_n}|} - \frac{1}{|\mb{x}-\mb{c_n}|} \right) + (\frac{1}{4\pi rQ} +1) \mathcal{O}(\omega^2)  \\
=& \frac{r}{Q} \left(M \cdot \partial_{\mb{x_1}}u^i(\mb{x^*}) + \frac{\omega^2}{4\pi r}\sum_{n=0}^\infty q_n  \int_{B_1\cup B_2}u^i \left(\frac{1}{|\mb{x}+\mb{c_n}|} - \frac{1}{|\mb{x}-\mb{c_n}|} \right) + \left(\frac{1}{4\pi r^2} + \frac{Q}{r}\right) \mathcal{O}(\omega^2)\right)\\
=& \frac{r}{Q} \left(M \cdot \partial_{\mb{x_1}}u^i(\mb{x^*}) + \frac{\omega^2}{4\pi r^2}\sum_{n=0}^\infty q_n a_n \int_{B_1\cup B_2}u^i(\mb{x}) + \left(\frac{1}{4\pi r^2} + \frac{Q}{r}\right) \mathcal{O}(\omega^2)\right) \\
=& \frac{r}{Q} \left(M \cdot \partial_{\mb{x_1}}u^i(\mb{x^*}) + \frac{A \omega^2}{4\pi r^2}  \int_{B_1\cup B_2}u^i(\mb{x}) + \left(\frac{1}{4\pi r^2} + \frac{Q}{r}\right) \mathcal{O}(\omega^2)\right) ,
\end{aligned}
\end{equation}
where $\mb{x^*}$ is some point lying in the line between the centers of $B_1$ and $B_2$. $A =\sum_{n=0}^\infty q_n a_n$ has the order of $Q$ since $a_n$  is uniformly bounded. It can be shown through the transformation between spherical coordinate system and Cartesian coordinate system. $M = \sum_{n=0}^\infty q_n p_n$ is some bounded constant independent of $\epsilon$ and $\alpha$. In fact,
\begin{equation}\label{M}
M =\sum_{n=0}^\infty q_n p_n = \sum_{n=0}^N q_n p_n + \sum_{n=N+1}^\infty q_n p_n = \sum_{n=1}^{N+1} \frac{1}{n^2} + \frac{1}{n}\mathcal{O}(\sqrt{\delta}) +\mathcal{O}(1) = \frac{\pi^2}{6}\mathcal{O}(1).
\end{equation}
Finally, the result in (\ref{cc}) together with Lemma (\ref{lemma2}) yields that
\begin{equation}
\begin{aligned}
\|\nabla u \|_{L^\infty(\Omega \backslash{\overline{B_1 \cup B_2}})} &\sim \left|\frac{\lambda_1 -\lambda_2}{C_1-C_2}  \right| \|\nabla h_\omega \|_{L^\infty(\Omega \backslash{\overline{B_1 \cup B_2}})} + \mathcal{O}(1) \\
&\sim \frac{r^*}{\epsilon^{1-\alpha}\left((1-\alpha)|\ln \epsilon|+\mathcal{O}(1) \right)}\left(M \cdot \partial_{\mb{x_1}}u^i(\mb{x^*}) + \frac{A \omega^2}{4\pi r^2}  \int_{B_1\cup B_2}u^i(\mb{x}) \right.\\
& \hspace{105pt}\left. + \left(\frac{1}{4\pi r^2} + \frac{Q}{r}\right) \mathcal{O}(\omega^2)\right) + \mathcal{O}(1),
\end{aligned}
\end{equation}
where we have used the estimate of $Q$ for $\alpha <1$.
On the other hand, if $\alpha \geq 1$, it is clear that $|C_1 -C_2| = \mathcal{O}(\epsilon^{-\alpha})$ and $|\lambda_1-\lambda_2| = \mathcal{O}(\epsilon^\alpha)$. One can derive that
\begin{equation}
\begin{aligned}
\|\nabla h_\omega (\mb{x})\|_{L^\infty (\bbr^3/\overline{D})}
\leqslant&\frac{1}{2\pi Q}\left[\sum_{m=0}^\infty \frac{q_m}{|\epsilon-c_m|^2} \right]  +\mathcal{O}(\omega^2)  = \mathcal{O}(\epsilon^{\alpha^2})+\mathcal{O}(\omega^2).
\end{aligned}
\end{equation}
Thus $\|\nabla u \|_{L^\infty(\Omega \backslash{\overline{B_1 \cup B_2}})}$ is uniformly bounded.

The proof is complete.

\section*{Acknowledgment}
The work of Y. Deng was supported by NSFC-RGC Joint Research Grant No. 12161160314 and NSF grant of China No.11971487.
The work was supported by the Hong Kong RGC General Research Funds (projects 12302919, 12301420 and 11300821),  the NSFC/RGC Joint Research Fund (project N\_CityU101/21), the France-Hong Kong ANR/RGC Joint Research Grant, A-HKBU203/19.


\end{document}